\documentclass[12pt]{article}
\usepackage{amsmath,amssymb,amsthm}
\usepackage[margin=3cm]{geometry}
\usepackage{blindtext}
\usepackage{titlesec}
\title{Sections and Chapters}

\title{\bf$C^{*}$- properties of vector-valued Banach algebras}

\author{Maryam Aghakoochaki  $^1$, \thanks{2020 Mathematics Subject Classifcation. Primary: 46J05; Secondary: 46J10} ,  \and Ali Rejali $^2$, \thanks{Corresponding author}}

\date{
	$^1$Isfahan University  \\ \texttt{mkoochaki@sci.ui.ac.ir, Orcid: 0000-0002-3851-6550}\\%
	$^2$Department of Pure Mathematics, Faculty of Mathematics and Statistics, University of Isfahan, Isfahan 81746-73441, Iran \\ \texttt{rejali@sci.ui.ac.ir, Orcid: 0000-0001-7270-665X}\\[2ex]%
\today
}

\newcommand{\tnorm}[1]{{\left\vert\kern-0.25ex\left\vert\kern-0.25ex\left\vert #1 
    \right\vert\kern-0.25ex\right\vert\kern-0.25ex\right\vert}}

\theoremstyle{plain}
\newtheorem{thm}{Theorem}[section]
\newtheorem{pro}[thm]{Proposition}
\newtheorem{cor}[thm]{Corollary}
\newtheorem{lem}[thm]{Lemma}
\theoremstyle{definition}

\newtheorem{ex}[thm]{\bf Example}

\newtheorem{DEF}[thm]{\bf Definition}

\begin{document}
\maketitle

\begin{abstract}
		Let $X$ be a locally compact Hausdorff space,  and $A$ be a commutative semisimple Banach algebra over the scalar field $\mathbb{C}$.
The correlation between different types of  $\textup{BSE}$- Banach algebras $A$, and the Banach algebra $C_{0}(X, A)$  are assessed. It is found and approved that $C_{0}(X, A)$ is a $C^{*}$- algebra if and only if $A$ is so.
Furthermore, 
 $C_{b}(X, A)= C_{0}(X, A)$ if and only if  $X$ is compact.
		
\end{abstract}

\section{Introduction}
 In 1990, Takahasi and Hatori  in \cite{E6}  introduced and assessed the notion of $\textup{BSE}$- algebras. Next, several authors have studied this concept for various kinds of Banach algebras; see \cite{AK}, \cite{AK2}, \cite{SB}, \cite{N2}, \cite{K1}. 
For more details see \cite{DU}, \cite{DU2}. In 2007, Inoue and Takahasi in \cite{F1} invented a new method for specifying the image of Gelfand representation of a commutative Banach algebra, in terms of quasi-topology.
They introduced the concept of $\textup{BED}$- concerning Doss, where Fourier-Stieltjes transforms of absolutely continuous measures are specified; \cite{RD1} and \cite{RD2}.
The Bochner-Eberlein-Doss  ($\textup{BED}$) is derived from the famous
theorem proved in \cite{RD1} and \cite{RD2}. They revealed that if $G$ is a locally compact Abelian
group, then the group algebra $L_{1}(G)$  is a $\textup{BED}$- algebra. Later \cite{AP}, the researcher revealed that $l^{p}(X, A)$ is a $\textup{BED}$- algebra if and only if $A$ is so.

The basic terminologies and the related information on  $\textup{BED}$- algebras are
extracted from \cite{F1}, \cite{kan}, and \cite{klar}.

  In this paper, $A$ is a commutative Banach algebra with the dual space $A^{*}$.
Let  $\Delta(A)$ be the character space of $ A$ with the Gelfand topology. $\Delta(A)$ is the set consisting of all non-zero multiplicative
linear functionals on $A$.  Assume that $C_{b}(\Delta(A))$ is the space consisting of all complex-valued continuous and bounded functions
on $\Delta(A)$, with sup-norm $\|.\|_{\infty}$ and pointwise product. The Gelfand map $A\to C_{b}(\Delta(A))$ is a  continuous algebra homomorphism on  $A$, where $a\mapsto \hat a$ such that $\hat{a}(\varphi)= \varphi(a)$; ($\varphi\in \Delta(A)$).
  If the Banach algebra $A$ is  semisimple, then the Gelfand map $\Gamma_{A}: A\to\widehat A$, $a\mapsto \hat a$, is injective or equivalently:
\begin{align*}
\underset{\varphi\in\Delta(A)}{\bigcap}\ker(\varphi)=\{0\}.
\end{align*}
Let $A$ be a semisimple Banach algebra. Then $\Delta(A)$ is compact if and only if $A$ is unital; See \cite{kan}.
A continuous linear operator $T$ on $A$ is named a multiplier if for all $x,y\in A$, $T(xy)=xT(y)$.
The set of all multipliers on $A$ will be denoted by $M(A)$. It is obvious that $M(A)$ is a Banach algebra, and if $A$ is an unital Banach algebra, then 
$M(A)\cong A$; See \cite{kan}.  As observed in   \cite{klar},  for each $T\in M(A)$ there exists a unique bounded  continuous function $\widehat{T}$ on $\Delta(A)$ expressed as:
$$\varphi(Tx)=\widehat{T}(\varphi)\varphi(x),$$
for all $x\in A$ and $\varphi\in \Delta(A)$.
By setting $\{\widehat{T}: T\in M(A)\}$, the $\widehat{M(A)}$ is yield.  Put ${\cal {M}}(A) :=\{ \sigma\in C_{b}(\Delta(A)): ~ \sigma.\hat{A}\subseteq \hat A\}$. Then $\widehat{M(A)}= \cal{M}(A)$, whenever $A$ is a semisimple Banach algebra. 
 Let $A$ be a Banach algebra. Then 
a bounded complex-valued continuous function $\sigma$ on $\Delta(A)$  is named a $\textup{BSE}$- function, if there exists a positive real number $\beta$ in a sense that for every finite complex number $c_{1},\cdots,c_{n}$,  and  the same many $\varphi_{1},\cdots,\varphi_{n}$ in $\Delta(A)$ the following inequality
$$\mid\sum_{i=1}^{n}c_{i}\sigma(\varphi_{i})\mid\leq \beta\|\sum_{i=1}^{n}c_{i}\varphi_{i}\|_{A^{*}}$$
holds.\\

The set of all $\textup{BSE}$- functions is expressed by $C_{\textup{BSE}}(\Delta( A))$, where for 
 each $\sigma$, the $\textup{BSE}$- norm of $\sigma$, $\|\sigma\|_{\textup{BSE}}$ 
is the infimum of all  $\beta$s  applied in the above inequality.   That $(C_{\textup{BSE}}(\Delta(A)), \|.\|_{\textup{BSE}})$ is a semisimple Banach subalgebra of $C_{b}(\Delta(A))$ is in Lemma 1 proved in \cite{E6}. The algebra $A$ is named a $\textup{BSE}$- algebra  if it meets the following condition:
$$ C_{\textup{BSE}}(\Delta(A)) = \widehat{M(A)}.$$ 
If $A$ is unital, then $\widehat{M(A)}= \widehat{A}\mid_{\Delta(A)}$, indicating that $A$ is a $\textup{BSE}$- algebra if and only if  $C_{\textup{BSE}}(\Delta(A)) = \widehat{A}\mid_{\Delta(A)}$. In this paper, by $\widehat A$ we means $ \widehat{A}\mid_{\Delta(A)}$.
The semisimple Banach algebra $A$ is named a norm- $\textup{BSE}$ algebra if there exists some  $K>0$ in a sense that for each $a\in A$, the following holds: 
$$
\|a\|_{A}\leq K\|\hat{a}\|_{\textup{BSE}};
$$
See\cite{DU}.
In general, 
\begin{align*}
\|\widehat a\|_{\infty} &\leq \|\widehat a\|_{\textup{BSE}} \leq \|\widehat a\|_{A^{**}}= \|a\|_{A}      
\end{align*}
for each $a\in A$.
The function $\sigma\in C_{BSE}(\Delta(A))$ is  a ${\textup{BED}}$- function, if for all $\epsilon>0$, there exists some compact set such $K\subseteq \Delta(A)$, where  for all $c_{i}\in\mathbb C$ and for all $\varphi_{i}\in\Delta(A)\backslash K$ the following inequality:
$$
|\sum_{i=1}^{n} c_{i}\sigma(\varphi_{i}) |\leq \epsilon\|\sum_{i=1}^{n}c_{i}\varphi_{i}\|_{A^{*}}
$$
holds. This definition of $\textup{BED}$-  functions is a modification of the definition in \cite{RD1}.
  The set of all $\textup{BED}$-  functions is expressed by the $C_{BSE}^{0}(\Delta(A))$.  Clearly, $C_{BSE}^{0}(\Delta(A))$ is a closed ideal of $ C_{BSE}(\Delta(A))$; see[\cite{F1}, Corollary 3.9].

  In \cite{kan}, it is proved that 
$$
C_{0}(X, A)= C_{0}(X)\check\otimes A
$$
as two Banach algebras which are isomorphic, where $C_{0}(X, A)$ is the Banach algebra of all continuous functions $f: X\to A$ which vanish at infinity. $C_{0}(X, A)$ with the following norm
$$
\|f\|_{\infty, A}= sup\{\|f(x)\|_{A}: ~ x\in X\}
$$
is a commutative Banach algebra with a pointwise product. It is proven that  $C_{0}(X, A)$ is a $\textup{BSE}$- algebra of type I if and only if $A$  is a $\textup{BSE}$- algebra of type I; See \cite{ARS}.
In the following, we define several types of $\textup{BSE}$- algebras that are used in this paper.
\begin{DEF}
Let $A$ be a commutative  Banach algebra where $\Delta(A)$ is non-empty. Then:\\
(i) $A$ is a $\textup{BSE}$- algebra if and only if $ C_{\textup{BSE}}(\Delta(A)) = \widehat{M(A)}$.\\
(ii) $A$ is a weak- $\textup{BSE}$ algebra if and only if $C_{\textup{BSE}}(\Delta(A)) = \hat A$.\\
(iii) $A$ is a  $\textup{BED}$- algebra if and only if $C_{\textup{BSE}}^{0}(\Delta(A)) = \hat{A}$.\\
(iv)  $A$ is a weak- $\textup{BED}$- algebra if and only if $C_{\textup{BSE}}^{0}(\Delta(A)) = \widehat{M(A)}$.
\end{DEF}
In the following, we characterize different type of   $\textup{BSE}$- algebras for commutative $C^{*}$- algebra $ C_{0}(X)$.
\begin{ex}
Assume that $A= C_{0}(X)$. Then \\
(i) $C_{\textup{BSE}}(\Delta(A)) = {C_{b}(X\widehat)}$ and so $A$ is a $\textup{BSE}$- algebra.\\
(ii)  $A$ is a weak-$\textup{BSE}$ algebra  if and only if  $X$ is compact.\\
(iii)  $C_{\textup{BSE}}^{0}(\Delta(A)) ={C_{0}(X \widehat)}$ and so $A$ is a $\textup{BED}$- algebra.\\
(iv)  $A$ is a weak- $\textup{BED}$ algebra if and only if   $X$ is compact.
\end{ex}
In this paper, $\textup{BSE}$-  types and $\textup{BED}$- types of  vector-valued $C^{*}$- algebra $C_{0}(X, A)$ will be investigated. It will be shown that:
\begin{thm}
Let $A$ be a commutative semisimple  Banach algebra and $X$ be a locally compact Hausdorff space. Then \\
(i) $C_{0}(X, A)$ is a ${\textup{BSE}}$-  algebra if and only if $A$ is a  ${\textup{BSE}}$- algebra.\\
(ii) $C_{0}(X, A)$ is a weak- ${\textup{BSE}}$  algebra if and only if $A$ is a  weak- ${\textup{BSE}}$ algebra and $X$ is compact.\\
(iii) $C_{0}(X, A)$ is a ${\textup{BED}}$-  algebra if and only if $A$ is a  ${\textup{BED}}$- algebra.\\
(iv) $C_{0}(X, A)$ is a weak- ${\textup{BED}}$  algebra if and only if $A$ is a  weak- ${\textup{BED}}$ algebra and $X$ is compact.
\end{thm}

We show that:
$$
C_{b}(X, A)= C_{0}(X, A)
$$
if and only if $X$ is compact.  It is found and approved that
$ C_{0}(X, A)$ is a $C^{*}$- algebra if and only if $A$ is so. Furthermore, new criteria are presented to make Banach algebras as a $C^{*}$- algebra.
\section{$C^{*}$- properties of $C_{0}(X, A)$}
Let $A$ be a commutative semisimple Banach algebra. In \cite{DM}, the author studies the vector-valued group algebra $L^{1}(G, A)$ to be a group algebra. He showed that $L^{1}(G, A)$ is a group algebra if and only if $A$ is a group algebra.
Furthermore, if $A$ has Radon- Nikodym property, then $M(G, A)$ is a measure algebra on a group if and only if $A$ is a measure algebra.
The researcher in \cite{MR} gives a characterization of commutative   Banach algebra which are the group algebras on locally compact Abelian groups.   Let $X$ be a locally compact Hausdorff space.
In this section, 
we show that $C_{0}(X, A)$ is a $C^{*}$- algebra, if and only if $A$ is so. In addition, new criteria are presented to make semisimple commutative Banach algebras as a $C^{*}$- algebra.
\begin{DEF}
Let $X$ be a Banach space and $(X, \leq)$ be a totally ordered set such that $X$ is a linear- lattice. Then\\
(i) If $\|a\vee b\|_{X}= {\textup{max}} \{\|a\|_{X}, \|b\|_{X}\}$, for all $a\geq 0$ and $b\geq 0$, then $X$ is called $M$- lattice.\\
(ii) If $\|a+ b\|_{X}= \|a\|_{X}+ \|b\|_{X}$, for all $a\geq 0$ and $b\geq 0$, then $X$ is called $L$- lattice.
\end{DEF}
We now mention the following definition of $L_{1}$- inducing character a Banach algebra which studied in \cite{DM}.
\begin{DEF}
Let $A$ be a commutative semisimple Banach algebra and $\varphi\in\Delta(A)$. Then $\varphi$ is named inducing group algebra if the following conditions hold:\\
(i) $\|\varphi\|_{A^{*}}=1$\\
(ii) Let  $P_{\varphi} :=\{a: \varphi(a)= \|\varphi\|\}$. Then $P_{\varphi}$ is a $L$- lattice.\\ 
(iii) If  $R_{\varphi}:= \{ a- b:  a,b\in P_{\varphi}\}$ for all  $a, b\in R_{\varphi}$ with $a\wedge b=0$. Then 
$$
\|a+ b\|_{A}= \|a- b\|_{A}
$$
(iv) For each $a\in A$ there exist unique $a_{1}, a_{2}\in R_{\varphi}$ where  $a= a_{1}+ia_{2}$. We put
$$
Re(a)= a_{1}, ~~ Im(a)= a_{2}
$$
(v) Let $|a|:= \textup{sup}\{Re(e^{i\Theta}a): ~ 0\leq \Theta\leq 2\pi\}$. Then $|a|\in A$ and $\| ~|a|~ \|_{A}= \|a\|_{A}$.\\
The set of all $L_{1}$- inducing characters of Banach algebra characterized by $A$ is denoted by $d(A)$.
\end{DEF}
Now, we will state some practical theorems and results; See\cite{DM}.
\begin{thm}
Let $A$ be a commutative semisimple Banach algebra and $G$ be an Abelian locally compact group. Then\\
(i) $A$ is a group algebra if and only if $A$ is Touberian and
$$
d(A)= \Delta(A).
$$
(ii) The algebra $L^{1}(G, A)$ is a group algebra if and only if 
$$
d(A)= \Delta(A).
$$
(iii)
$$
d(L^{1}(G, A)) = \Delta(L^{1}(G, A))
$$
if and only if 
$$
\Delta(A)= d(A)
$$
\end{thm}
\begin{DEF}
Let $A$ be a commutative semisimple Banach algebra and  $\varphi\in \Delta(A)$. Then $\varphi$ is called  $C^{*}$-  inducing character if the following conditions hold:\\
(i)
\begin{align*}
  \|\varphi\|_{ A^{*}}=1
\end{align*}
(ii) Let
\begin{align*}
  P_{\varphi}:= \{a\in A: \varphi(a)\geq 0\}. 
\end{align*}
Then $P_{\varphi}$ is a $M$- lattice.\\
(iii) If 
$$
R_{\varphi}:= \{a- b: a, b  \in P_{\varphi}\}
$$
and  $a, b\in R_{\varphi}$
such that $a\wedge b=0$. Then 
$$
\|a+b\|_{ A}= \|a-b\|_{ A}
$$
(iv) For each $a\in  A$ there exist unique $a_{1},a_{2}\in R_{\varphi}$ such that $a= a_{1}+i a_{2}$. We denote that $a_{1}= Re(a)$ and $a_{2}= Im(a)$.\\
(v) Let $|a|:= a\vee (-a)$. Then 
$|a|\in A$ and $\| ~|a|~ \|_{A}= \|a\|_{ A}$ for each $a\in  A$.
The set of all $C^{*}$-  inducing characters  of $A$ is denoted by $D(A)$.
\end{DEF}
\begin{lem}\label{ldc1}
Let $X$ be a locally compact Hausdorff space and $A$ be a  semisimple Banach algebra such that $D(A)= \Delta(A)$. Then $D(C_{0}(X, A))=  \Delta(C_{0}(X, A))$.
\end{lem}
\begin{proof}
(i)
Let $x\in X$ and $\varphi\in D(A)$, then $x\otimes \varphi\in D((C_{0}(X, A)))$ and $\|x\otimes\varphi\|= \|\varphi\|=1$. In fact,
if $\varphi\in\Delta(A)$, then 
\begin{align*}
1 &= \|\varphi\|= \|\delta_{x}\|.\|\varphi\|\\
     &= \|\delta_{x}\otimes\varphi\|:= \|x\otimes \varphi\|\leq\|\delta_{x}\|.\|\varphi\|\leq\|\varphi\|=1
\end{align*}
Thus $\|x\otimes\varphi\|=1$.\\
 (ii) $P_{x\otimes\varphi}$ is a $M$- lattice. If $f,g\in P_{x\otimes\varphi}$, then
\begin{align*}
\|f\vee g\|_{\infty, A} &= sup\{\|f\vee g(x)\|_{A} : x\in X\}\\
                                     &=  sup\{\|f(x)\vee g(x)\|_{A} : x\in X\}\\
                                     &=  sup ~max\{\|f(x)\|_{A}, \|g(x)\|_{A} : x\in X\}\\
                                      &= max~ sup\{\|f(x)\|_{A}, \|g(x)\|_{A} : x\in X\}\\
                                        &= max\{\|f(x)\|_{\infty,A}, \|g(x)\|_{\infty,A} : x\in X\}\\
                                         &= \|f\|_{\infty,A}\vee \|g\|_{\infty,A}
\end{align*}
(iii)
If $f,g \in R_{x\otimes\varphi}$ and $f\wedge g =0$,  so $f(x)\wedge g(x)= 0$, for each $x\in X$. Then 
\begin{align*}
\|f+g\|_{\infty, A} &= sup\{\|f(x)+ g(x)\|_{A} : x\in X\}\\
                                  &= sup\{ \|f(x)-g(x)\|_{A}: x\in X\}\\
                                   &= \|f- g\|_{\infty, A}   
\end{align*}
(iv) If $f\in C_{0}(X, A)$, and $y\in X$, then there exist $a_{y}^{1}, a_{y}^{2}\in R_{\varphi}$ such that $f(y)= a_{y}^{1}+i a_{y}^{2}$. Define $f_{1}(y)= a_{y}^{1}$ and $f_{2}(y)= a_{y}^{2}$, for all $y\in X$. Thus 
$f= f_{1}+ if_{2}$ and $f_{1}, f_{2}\in R_{x\otimes\varphi}$.\\
(v)  If $f\in C_{0}(X, A)$, then $|f|(x)= f(x)\vee (-f(x))$ and $a\leq b$ if and only if $\varphi(a-b)\geq 0$ so $(A, \leq)$ is ordered set. The following is the yield:
\begin{align*}
\|~|f|~\|_{\infty, A} &= sup \{\| ~|f|(x)~ \|_{A} : ~ x\in A\}\\
                                   &=  sup \{\| ~|f(x)|~ \|_{A} : ~ x\in A\}\\
                                   &=   sup \{\| ~|f(x)|~ \|_{A} : ~ x\in A\}\\
                                  &= \|f\|_{\infty, A}
\end{align*}
\end{proof}
\begin{lem}\label{ldc2}
 Let $X$ be a locally compact Hausdorff space and $A$ be a  semisimple Banach algebra such that $D(C_{0}(X, A))= \Delta(C_{0}(X, A))$. Then $D(A) = \Delta(A)$ and conversely.
\end{lem}
\begin{proof}
(i) Let $\varphi\in\Delta(A)$ and $x_{0}\in X$. Then 
$$
x_{0}\otimes \varphi\in \Delta(C_{0}(X, A))= D(C_{0}(X, A)).
$$
By applying Urysohn's lemma there exits $f_{0}\in C_{0}(X)$ such that $0\leq f_{0}\leq 1$ and $f_{0}(x_{0})=1$. We have
\begin{align*}
\|x_{0}\otimes\varphi\|_{\infty, A} &\geq |x_{0}\otimes\varphi(f_{0}\otimes a)|\\
                                                          &= f_{0}(x_{0})|\varphi(a)|= |\varphi(a)|
\end{align*}
for all $a\in A$, so $\|x_{0}\otimes\varphi\|_{\infty, A} \geq \|\varphi\|_{A^{*}}$. Moreover, 
$$\|x_{0}\otimes\varphi\|_{\infty, A} \leq \|\delta_{x_{0}}\|.\|\varphi\|_{A^{*}}\leq \|\varphi\|_{A^{*}}.$$
 Thus $\|\varphi\|=\|x_{0}\otimes\varphi\|=1$.\\
(ii) Let $a,b\in A$ where $a\wedge b=0$, $x_{0}\in X$ and $f_{0}\in C_{0}(X)$ be a non-zero and $\|f_{0}\|= 1$. Then $P_{x_{0}\otimes\varphi}$  is a $M$- lattice, so 
$$
\|f\vee g\|_{\infty, A}= \|f\|_{\infty, A}\vee \|g\|_{\infty, A}
$$
In a special case, put $f= f_{0}\otimes a$ and $g= f_{0}\otimes b$ we have
\begin{align*}
\|a\vee b\|_{A} &= \|f_{0}\otimes a \vee f_{0}\otimes b\|_{\infty, A}\\
                           &= \|f_{0}\|.(\|a\|_{A}\vee \|b\|_{A})= \|a\|_{A}\vee \|b\|_{A}
\end{align*}
Then $P_{\varphi}$ is a $M$- lattice.\\
(iii) Let $a\in A$ and $x_{0}\in X$. Choose $f\in C_{0}(X)$ where $f(x_{0})= 1$. Clearly, $f\otimes a\in C_{0}(X, A)$, thus there exist unique $f_{1},f_{2}\in R_{x_{0}\otimes\varphi}$ such that $f= f_{1}+if_{2}$. Put
$a_{1}= f_{1}(x_{0})$ and $a_{2}= f_{2}(x_{0})$, so $a= a_{1}+ia_{2}$. Now, we show that $a_{1}, a_{2}\in R_{\varphi}$. Since $f_{1}\in  R_{x_{0}\otimes\varphi}$, so there exist $g_{1},g_{2}\in  P_{x_{0}\otimes\varphi}$ where
$f_{1}= g_{1}-g_{2}$, $\varphi(g_{1}(x_{0})\geq 0$ and $\varphi(g_{2}(x_{0})\geq 0$. Thus  $b_{1}= g_{1}(x_{0})$, $b_{2}= g_{2}(x_{0})$ are
in $P_{\varphi}$. Then, $a_{1}= b_{1}- b_{2}\in R_{\varphi}$. Similarly $a_{2}\in R_{\varphi}$.\\

(iv) Let $a, b\in P_{\varphi}$ and  $f\in C_{0}(X)$ where $f(x_{0})= 1$. 
Thus 
$$\varphi(a\vee b)= x_{0}\otimes \varphi(f\otimes a\vee f\otimes b)\geq 0$$
 Similarly, $\varphi(a\wedge b)= x_{0}\otimes \varphi(f\otimes a\wedge f\otimes b)\geq 0$.
Therefore, $a\wedge b, a\vee b\in P_{\varphi}$, so $P_{\varphi}$ is a linear lattice.
 Thus $D(A)= \Delta(A)$.
  
       Conversely, let  $D(A)= \Delta(A)$. Then according to Lemma \ref{ldc1} the following is yaild: 
\begin{align*}
\Delta(C_{0}(X, A)) &= X\times \Delta(A)\\
                                   &= X\times D(A)\\
                                     &=D(C_{0}(X, A))
\end{align*}
(v)  Let $a\in A$ and $f\in C_{0}(X)$ such that $\|f\|_{\infty}=1$. Then 
\begin{align*}
\|a\|_{A} &= \|f\|_{\infty}.\|a\|_{A}= \|f\otimes a\|_{\infty, A}\\
                  &= \| ~|f\otimes a|~ \|_{\infty, A}= \| ~|f|~ \|_{\infty}. \| ~|a|~ \|_{A}\\
                     &= \|f\|_{\infty}.  \| ~|a|~ \|_{A}=  \| ~|a| ~\|_{A}
\end{align*}
 Thus $ \| ~|a|~ \|_{A}= \|a\|_{A}$.\\

\end{proof}
\begin{pro}\label{ldc3}
Let $X$ be a locally compact Hausdorff space and $A$ be a commutative semisimple Banach algebra. Then $C_{0}(X, A)$ is a $C^{*}$- algebra if and only if $A$ is a $C^{*}$- algebra.
\end{pro}
\begin{proof}
Assume that $A$ is a $C^{*}$- algebra and $A= C_{0}(Y)$, where $Y= \Delta(A)$. Then
\begin{align*}
C_{0}(X, A) &= C_{0}(X)\check\otimes A\\
                     &= C_{0}(X)\check\otimes  C_{0}(Y)\\
                       &=  C_{0}(X\times Y)
\end{align*}
is a $C^{*}$- algebra.

     Conversely, if  $C_{0}(X, A)$ is a $C^{*}$- algebra. Since $A$ is  commutative, so $C_{0}(X, A)$ is commutative and by using Gelfand Theorem the following is yield:
\begin{align*}
C_{0}(X, A) = C_{0}(\Delta(C_{0}(X, A)))= C_{0}(X\times \Delta(A))
\end{align*}
Thus 
$$
C_{0}(X, A) \cong C_{0}(X)\check\otimes  C_{0}(\Delta(A))
$$
In the sequel, we prove that $ C_{0}(\Delta(A))\cong \hat A\cong A$. Let $g\in  C_{0}(\Delta(A))$, $x_{0}\in X$ and $f\in  C_{0}(X)$ such that $f(x_{0})=1$, then $f\otimes g\in C_{0}(X, A) $. Thus there exist a sequence $(h_{n})$ such that
$f\otimes g = \underset{n}{lim}h_{n}$, where $h_{n}\in  C_{0}(X)\otimes A$. 
Then
\begin{align*}
\hat{g}(\varphi) &= \varphi(g) = x_{0}\otimes \varphi(f\otimes g)\\ 
                           &=  \underset{n}{lim} x_{0}\otimes \varphi(h_{n}) = \underset{n}{lim}\varphi(h_{n}(x_{0}))\\ 
                           &= \varphi( \underset{n}{lim}h_{n}(x_{0}))= \hat{b}(\varphi)
\end{align*}
where $b :=  \underset{n}{lim}(h_{n}(x_{0}))$. Therefore $g= \hat{b}\in \hat A$, also by semisimpility of $A$, $A=\hat A$. Hence $C_{0}(\Delta(A))\subseteq \hat A$. Let $a\in A $, $\epsilon>0$ and $K:= \{\varphi\in \Delta(A): |\varphi(a)|\geq\epsilon\}$.
Then $K$ is $w^{*}$- closed set  in the unit ball of $A^{*}$, which is compact. Thus $|\varphi(a)|<\epsilon$ for all $\varphi\in\Delta(A)\backslash K$. Hence $\hat a$ vanishes at infinity. Clearly, $\|\hat a\|_{\infty}\leq\|a\|<\infty$
and if $\varphi_{\alpha}\overset{w^{*}}{\to}\varphi$, then $\hat{a}(\varphi_{\alpha})= \varphi_{\alpha}(a)\to \varphi(a)=\hat{a}(\varphi)$.
Therefore $\hat{a}\in C_{0}(\Delta(A))$ and so 
$$
\hat {A}\subseteq C_{0}(\Delta(A))
$$
Thus $C_{0}(\Delta(A))= \hat A$ is a $C^{*}$- algebra, so $A$ is a $C^{*}$- algebra.
\end{proof}
\begin{lem}\label{ldc4}
Let $A$ be a commutative semisimple Banach algebra. Then $A$ is a $C^{*}$- algebra if and only if 
$$
\Delta(A) = D(A)
$$
\end{lem}
\begin{proof}
If $A= C_{0}(X)$, then it is clear that  $\Delta(A) = D(A)$. 

  Conversely, if  $\Delta(A) = D(A)$, then for all $\varphi\in D(A)$ the space $A_{\varphi}:= \{a: ~ \varphi(a)\geq0\}$ is a $M$- lattice. Thus  due to Kakutani's theorem  $A_{\varphi}$  is isometrically embedding  in $C(X_{\varphi})$ for a  Hausdorff compact set 
  $X_{\varphi}$; See \cite{BK}. 
Take $X$ as the product of $X_{\varphi}$, so $X$ is compact and we show that $A$ is a closed subalgebra of $C(X)$. \\
We consider two cases:\\
(i) Assume that $A$ is unital. Thus for all $a\in A$, there exist some maximal ideal $M= ker(\varphi)$ where 
$$
a\in I= \{ ab| b\in A\}\subseteq ker(\varphi)
$$ 
for some $p\in \Delta(A)$, so $A=\underset{\varphi\in\Delta(A)}{\bigcup}ker(\varphi)$. Since $\Delta(A)= D(A)$, so 
$$
ker(\varphi)\subseteq A_{\varphi}= \{a: \varphi(a)\geq 0\}
$$
is a linear lattice and
$$
A= \underset{\varphi\in\Delta(A)}{\bigcup}ker(\varphi) = \underset{\varphi\in\Delta(A)}{\bigcup} A_{\varphi}
$$
If $\varphi\leq \psi$, then $A_{\varphi}\subseteq A_{\psi}$. If 
$$
id_{\varphi, \psi}: A_{\varphi} \to A_{\psi}
$$
is identity map, then for $\varphi\leq \psi \leq  \eta$ we have
$id_{\varphi, \eta}= id_{\varphi, \psi}oid_{\psi, \eta}$. Thus 
$$
A= \underset{\longleftarrow}{lim}A_{\varphi}
$$
By applying  Kakutani's theorem there exists some Hausdorff compact set $X_{\varphi}$ such that $A_{\varphi}$ is isometrically embedded in $C(X_{\varphi})$. So $A$ is closed subset of $C^{*}$- algebra $C(X)$, where $X= \underset{\varphi}{\prod}X_{\varphi}$. Therefore $A$ is a $C^{*}$- algebra.

  (ii) If $A$ has no identity, then
$$
D(A_{1}) = \Delta(A_{1})= \Delta(A)\cup \{\varphi_{0}\}
$$
where $\varphi_{0}(a, \lambda)= \lambda$ and $A_{\varphi_{0}}= A\times {\mathbb{R}^{+}}$, then according to part (i)  $A_{1}= A\oplus \mathbb C$ with a suitable norm is a $C^{*}$- algebra. Since $A$ is a closed ideal in $A_{1}$ , so $A$ is a $C^{*}$- algebra.
\end{proof}
We now state the main result of this section. By using Lemmas \ref{ldc1}- \ref{ldc4}, the following is immediate.
\begin{thm}
Let $X$ be a locally compact Hausdorff space and $A$ be a commutative semisimple Banach algebra. Then the following expressions are equivalent\\
(i) $C_{0}(X, A)$ is a $C^{*}$- algebra. \\
(ii) $\Delta(A)= D(A)$.\\
(iii) $A$ is a $C^{*}$- algebra.\\
(iv) $\Delta(C_{0}(X, A)) = D(C_{0}(X, A))$.
\end{thm}
\begin{ex}\cite{JT} \\
(i) $d (M(G))= \Delta(L^{1}(G))= \hat G$.\\
(ii) $D(C_{0}(X))= \Delta(C_{0}(X))= X$
\end{ex}

\section{$\textup{BSE}$- properties of $C_{0}(X, A)$}
Let $A$ be a commutative Banach algebra with nonempty character space and $X$ be a locally compact Hausdorff topological space. 
Let $f\in C(X)$ and $a\in A$. Define $f\otimes a: X\to A$ where $f\otimes a(x):= f(x)a$. So if $f\in C_{0}(X)$, then $f\otimes a\in C_{0}(X, A)$, and if $f\in C_{b}(X)$, then $f\otimes a\in C_{b}(X, A)$.
 It is obvious that
$$
C_{0}(X, A)= C_{0}(X)\check\otimes A
$$
Let $x\in X$, $\varphi\in \Delta(A)$ and $f\in C_{b}(X, A)$. Then $x\otimes\varphi(f)= \varphi(f(x))$ and 
$$
\Delta(C_{0}(X, A))= \{x\otimes \varphi: x\in X ~, ~\varphi\in \Delta(A)\} 
$$
In the following, we state the main properties of the space $C_{0}(X, A)$ which is required in this section.
\begin{lem}\label{l1}
Let $A$ be a commutative Banach algebra and  $X$ be a locally compact Hausdorff topological space. Then\\
(i) $C_{0}(X, A)$ is a Banach algebra under pointwise product.\\
(ii) $C_{0}(X, A)$ has a bounded $\Delta$- weak approximate identity if and only if $A$ has.\\
(iii) $C_{0}(X, A)$ has a bounded approximate identity if and only if $A$ has.\\
(iv) $C_{0}(X, A)$  is semisimple if and only if $A$ is semisimple.
\end{lem}
\begin{proof}
For proofs of (i) and (ii), see \cite{AFR} and for (ii) see \cite{ARC}. Since $C_{0}(X)$ has approximate property, so by using (iv) \cite{JT} is held. 
\end{proof}

\begin{lem}
Let $A$ be a commutative Banach algebra with nonempty character space and $X$ be a locally compact Hausdorff topological space. Then
$$
A^{*}o C_{0}(X, A)= C_{0}(X).
$$
\end{lem}
\begin{proof}
If $f\in C_{0}(X)$ and $a\in A$ is nonzero, then by Hahn- Banach theorem there exists $g\in A^{*}$ where $g(a)= 1$. Put $h= f\otimes a$, so  $h\in C_{0}(X, A)$ such that 
$$
goh(x)= g(f(x)a)= f(x)g(a)= f(x) ~~~~  (x\in X).
$$
Hence $f=goh$.
Therefore $C_{0}(X)\subseteq A^{*}o C_{0}(X, A)$.

   Conversely, if $g\in A^{*}$ and $h\in C_{0}(X, A)$, then for all $\epsilon>0$, there exist some compact set $K$ in $X$ where for all $x\in X\backslash K$ the following inequality holds:
$$
\|h(x)\|_{A}< \frac{\epsilon}{1+\|g\|}.
$$
Hence 
$$
|goh(x)|\leq \|g\|.\|h(x)\|_{A}<\epsilon
$$
for all $x\in X\backslash K.$
Thus $goh$ is the composition of two continuous functions and 
$$\|goh\|\leq\|g\|.\|h\| $$
 Therefore $A^{*}o C_{0}(X, A)\subseteq C_{0}(X)$. This completes
 the proof.
\end{proof}
\begin{cor}
Let $A$ be a commutative Banach algebra with nonempty character space and $X$ be a locally compact Hausdorff topological space. Then
$$
D(A)o C_{0}(X, A)= C_{0}(X).
$$
\end{cor}

\begin{lem}
Let $A$ be a commutative unital  Banach algebra.  Then 
$$
M(C_{0}(X, A))= C_{b}(X, A) 
$$
\end{lem}
\begin{proof}
If $f\in C_{0}(X, A)$ and $g\in C_{b}(X, A)$, then $f.g\in C_{0}(X, A)$. Define $L_{g}(f) := f.g$, it is obvious that $L_{g}\in M(C_{0}(X, A))$.\\
If $T\in M(C_{0}(X, A))$ and $g(x):= T(f_{x}\otimes 1_{A})$, where $f_{x}\in C_{0}(X)$ and $f_{x}(x)=1$, then for each $h\in C_{0}(X, A)$:
\begin{align*}
L_{g}(h)(x) &= g(x).h(x)= h(x).T(f_{x}\otimes 1_{A})(x)\\
                   &= T(h. (f_{x}\otimes 1_{A}))(x)
\end{align*}
Therefore $L_{g}(h)(x)= (f_{x}\otimes 1_{A}).T(h)(x)$ and so 
$$
L_{g}(h)(x)= (f_{x}\otimes 1_{A})(x)T(h)(x)= f_{x}(x)1_{A}.Th(x)= Th(x)
$$
Which implies that
$$
L_{g}(h)(x)= Th(x)
$$
for all $x\in X$. Then $L_{g}(h)= Th$ and so $L_{g}= T$. Consequently,
$$
M( C_{0}(X, A))\subseteq \{L_{g}: ~ g\in C_{b}(X, A)\}
$$
Therefore 
$$
M(C_{0}(X, A))= C_{b}(X, A) .
$$
\end{proof}
\begin{lem}\label{lcc}
$C_{b}(X, A)= C_{0}(X, A)$ if and only if $X$ is compact. 
\end{lem}
\begin{proof}
Suppose that $C_{b}(X, A)= C_{0}(X, A)$. 
If $f\in C_{b}(X)$ and $a\in A$, then $g= f\otimes a$ is in $C_{0}(X, A)$. Thus $g= lim_{n}h_{n}$ where $h_{n}\in C_{0}(X)\otimes A$. Therefore 
$$
g(x) = f(x)a = lim h_{n}(x).
$$
If $p\in A^{*}$ and $p(a)=1$, then 
$$
f(x)=  lim p(h_{n}(x)).
$$
Assume that $h_{n}= \sum_{i=1}^{l_{n}}k_{i}\otimes a_{i}$, then
$$
p(h_{n}(x))= \sum_{i=1}^{l_{n}}c_{i}k_{i}(x).
$$
where $c_{i}= p(a_{i})$.
So $poh_{n}\to f$ and $poh_{n}\in C_{0}(X)$, thus $f\in C_{0}(X)$, In fact,
$$
|p(h_{n}(x)-p(g(x))|\leq \|p\|.\|h_{n}(x)- g(x)\|_{A}
$$
Also, $\|poh_{n}- f\|_{\infty}\leq \|p\|. \|h_{n}-g\|_{\infty,A}$ and $poh_{n}\in C_{0}(X)$, Therefore $f\in C_{0}(X)$.
 and $C_{b}(X) = C_{0}(X)$. Therefore $X$ is compact.

   The converse is clear.
\end{proof}

\begin{lem}\label{lcbeta}
Let  $X$ be a locally compact Hausdorff topological space and $A$ be a commutative   Banach algebra. Then
$$
C_{b}(X, A) = C(\beta X, A)
$$
isometrically isomorphism as Banach algebras.
\end{lem}
\begin{proof}
This lemma is proved similar to Lemma 3.2 of   \cite{ARC}.
\end{proof}
\begin{lem}\label{l3v}
If $\varphi_{i}\in \Delta(A)$, $c_{i}\in\mathbb C$ and $x_{i}\in X$, Then\\
(i) 
$$
\|\sum_{i=1}^{n} c_{i}\varphi_{i}\|_{A^{*}}\leq  \|\sum_{i=1}^{n} c_{i}(x_{i}\otimes\varphi_{i})\|_{C_{0}(X, A)}^{*}
$$
(ii) Let $x\in X$. Then
$$
 \|\sum_{i=1}^{n} c_{i}(x\otimes\varphi_{i})\|_{{C_{0}(X, A)}^{*}}\leq \|\sum_{i=1}^{n} c_{i}\varphi_{i}\|_{A^{*}}
$$
(iii) Let $f\in C_{0}(X)$. Then
$$
 \|\sum_{i=1}^{n} c_{i}f(x_{i})\varphi_{i}\|_{A^{*}}\leq \|f\|_{\infty}\|\sum_{i=1}^{n} c_{i}(x_{i}\otimes\varphi_{i})\|_{{C_{0}(X, A)}^{*}}.
$$
\end{lem}
\begin{proof}
(i)
Let $f_{0}\in C_{0}(X)$ such that $\|f_{0}\|_{\infty}\leq 1$ and $f_{0}(x_{i})= 1$, for all $1\leq i\leq n$. Then
\begin{align*}
\|\sum_{i=1}^{n} c_{i}(x_{i}\otimes\varphi_{i})\|_{{C_{0}(X, A)}^{*}} &= \textup{sup}\{ |\sum_{i=1}^{n} c_{i} x_{i}\otimes\varphi_{i}(g)|:  ~\|g\|_{\infty, A}\leq 1\}\\
                                                                                                                     &\geq\textup{sup}\{ |\sum_{i=1}^{n} c_{i}(x_{i}\otimes\varphi_{i})(f\otimes a)|: ~ \|f\|_{\infty}\leq 1, \|a\|_{A}\leq 1\}\\
                                                                                                                      &\geq \textup{sup}\{ |\sum_{i=1}^{n} c_{i}(x_{i}\otimes\varphi_{i})(f_{0}\otimes a)|: ~  \|a\|_{A}\leq 1\}\\
                                                                                                                      &= \textup{sup}\{ |\sum_{i=1}^{n} c_{i}\varphi_{i}( a)|: ~  \|a\|_{A}\leq 1\}\\
                                                                                                                      &= \|\sum_{i=1}^{n} c_{i}\varphi_{i}\|_{A^{*}}
\end{align*}
(ii)
\begin{align*}
\|\sum_{i=1}^{n} c_{i}(x\otimes\varphi_{i})\|_{{C_{0}(X, A)}^{*}} &= \textup{sup}\{ |\sum_{i=1}^{n} c_{i}\varphi_{i}(g(x))|: ~ \|g\|_{\infty, A}\leq 1\}\\
                                                                                                                &= \textup{sup}\{ |{\widehat{g(x)}}(\sum_{i=1}^{n} c_{i}\varphi_{i})|:  ~\|g\|_{\infty, A}\leq 1\}\\
                                                                                                                  &\leq \textup{sup}\{\|g(x)\|_{A}.\|\sum_{i=1}^{n} c_{i}\varphi_{i}\|_{A^{*}}: ~ \|g\|_{\infty, A}\leq 1\}\\
                                                                                                                  &\leq \|\sum_{i=1}^{n} c_{i}\varphi_{i}\|_{A^{*}}
\end{align*}
(iii)
\begin{align*}
\|\sum_{i=1}^{n} c_{i}f(x_{i})\varphi_{i}\|_{A^{*}} &= \textup{sup}\{ |\sum_{i=1}^{n} c_{i} f(x_{i})\varphi_{i}( a)|: ~  \|a\|_{A}\leq 1\}\\
                                                                                       &= \textup{sup}\{ |\sum_{i=1}^{n} c_{i}(x_{i}\otimes\varphi_{i})(f\otimes a)|:  ~\|a\|_{A}\leq 1\}\\
                                                                                       &\leq \textup{sup}\{ \|f\|_{\infty}.\|a\|_{A}.\|\sum_{i=1}^{n} c_{i}(x_{i}\otimes\varphi_{i})\|_{{C_{0}(X, A)}^{*}}: ~ \|a\|_{A}\leq 1\}\\
                                                                                        &= \|f\|_{\infty} \|\sum_{i=1}^{n} c_{i}(x_{i}\otimes\varphi_{i})\|_{{C_{0}(X, A)}^{*}}
\end{align*}
\end{proof}
In this section, the correlations between the \textup{BSE}- property and \textup{BED}- property of Banach algebra $A$ and $C_{0}(X, A)$ are assessed
\begin{pro}\label{pfc}
Let  $X$ be a locally compact Hausdorff topological space and $A$ be a commutative unital  Banach algebra.
If $f\in C_{b}(X, A)$ and $\widehat{f}\in C_{\textup{BSE}}^{0}(\Delta(C_{0}(X, A)))$, then $f\in C_{0}(X, A)$.
\end{pro}
\begin{proof}
Let $f\in C_{b}(X, A)$ where $\widehat{f}\in C_{\textup{BSE}}^{0}(\Delta(C_{0}(X, A)))$. Then according to definition for all $\epsilon>0$ there exists some compact set $K\subseteq X\times \Delta(A)$ such that 
for all  $c_{i}\in \mathbb C$ and for all $\psi_{i}= x_{i}\otimes \varphi_{i}\notin K$ the following inequalitty is yaild:
$$
|\sum_{i=1}^{n}c_{i}\hat{f}( x_{i}\otimes \varphi_{i})|<\frac{\epsilon}{2}
$$
But due to Lemma \ref{lcbeta}, the following is the yield:
\begin{align*}
C_{b}(X, A) &=  C_{b}(X)\check\otimes A.
\end{align*}
Hence there exists some sequence $(f_{m})$ in $C_{0}(X)\otimes A$ such that $f_{m}\overset{\|.\|_{\epsilon}}{\to}f$ and so  $\widehat{f_{m}}\overset{\|.\|_{\infty}}{\to}\widehat f$.
This implies that 
$$
|\sum_{i=1}^{n}c_{i}\hat{f}( x_{i}\otimes \varphi_{i})|= lim_{m} |\sum_{i=1}^{n}c_{i}\hat{f_{m}}( x_{i}\otimes \varphi_{i})| <\frac{\epsilon}{2}. \|\sum_{i=1}^{n}c_{i}( x_{i}\otimes \varphi_{i})\|_{{C_{0}(X, A)}^{*}}
$$
Thus there exists some $N>0$ where for all $m\geq N$, we have:
$$
 |\sum_{i=1}^{n}c_{i}\hat{f_{m}}( x_{i}\otimes \varphi_{i})|<\epsilon . \|\sum_{i=1}^{n}c_{i}( x_{i}\otimes \varphi_{i})\|_{{C_{0}(X, A)}^{*}}
$$
Assume that $f_{m}:= \sum_{j=1}^{l_{m}} g_{j}\otimes a_{j}\in C_{b}(X)\otimes A$.
Let $k\in\mathbb N$ such that $\varphi_{k}(a_{k})=1$ and $\varphi_{k}(a_{j})=0$ for all $1\leq j\leq l_{m}$ where $j\neq k$ and $m\geq k+ N$. Put $\varphi_{i}=\varphi_{k}$ for $1\leq i\leq n$. Then 

\begin{align*}
|\sum_{i=1}^{n}c_{i}\widehat{g_{k}}(\delta_{x_{i}})| &< \epsilon \|\sum_{i=1}^{n} c_{i}(x_{i}\otimes \varphi_{k})\|_{{C_{0}(X, A)}^{*}}\\
                                                                          &< \epsilon \|\sum_{i=1}^{n} c_{i}\delta_{x_{i}}\|_{{C_{0}(X)}^{*}}        
\end{align*}
Hence $\widehat{g_{k}}\in C_{\textup{BSE}}^{0}(\Delta(C_{0}(X)))= C_{0}(X\widehat)$ and so $g_{k}\in C_{0}(X)$ for each $k\in\mathbb N$.
Therefore 
$$
f_{m}= \sum_{k=1}^{l_{m}} g_{k}\otimes a_{k}\in C_{0}(X)\otimes A
$$
and $f_{m}\overset{\|.\|_{\epsilon}}{\to} f$, thus $f\in C_{0}(X, A)$. 
\end{proof}
The Banach algebra $A$ is called ${\textup{BED}}$- norm algebra, if there exist $M> 0$ such that $\|a\|_{A}\leq M\|\hat a\|_{\infty}$ for all $a\in A$.
\begin{pro}\label{pbdn}
(i)
Let $A$ be a commutative semisimple  ${\textup{BED}}$- algebra. Then it is a ${\textup{BSE}}$- norm algebra.\\
(ii) Let $A$ be a commutative semisimple  ${\textup{BED}}$- norm algebra. Then 
$$
A\cong  C_{\textup{BSE}}^{0}(\Delta(A))
$$
isometrically isomorphism with equivalent norms.
\end{pro}
\begin{proof}
(i)
By hypothesis,
$\widehat{A}= C_{\textup{BSE}}^{0}(\Delta(A))$ is a Banach algebra, so $(\widehat {A}, \|.\|_{\textup{BSE}})$ is complete. But the map $a\to\hat a$ is a continuous and isometric, thus by applying the open mapping theorem it has a continuous inverse map.  
As a result, there exists $M> 0$ such that $\|a\|\leq M\|\hat a\|_{\textup{BSE}}$.\\
(ii)
If $A$ is a ${\textup{BED}}$- norm algebra, then
\begin{align*}
\|\widehat a\|_{\infty} &\leq \|\widehat a\|_{\textup{BSE}}\\
                                      &\leq \|a\|_{A}\leq M\|\widehat a\|_{\infty}.
\end{align*}
Thus $A= (\widehat{A}, \|.\|_{\infty})$ is a sub- $C^{*}$ algebra with equivalent norms. Since   $\textup{BSE}$- property and  $\textup{BED}$- property  are preserve under equivalent norms, so $A$ is a  $\textup{BSE}$  and $\textup{BED}$- algebra.
Therefore
$$
A=  C_{\textup{BSE}}^{0}(\Delta(A)).
$$
\end{proof}
\begin{lem}\label{l1cbse}
Let $A$ be a commutative semisimple unital $\textup{BSE}$- algebra. Then $C_{0}(X, A)$ is a $\textup{BSE}$-  algebra.
\end{lem}
\begin{proof}
Let $g\in C_{b}(X, A\widehat)$, then
\begin{align*}
|\sum_{i=1}^{n}c_{i}\hat{g}(x_{i}\otimes\varphi_{i})| &= |\sum_{i=1}^{n}c_{i}\varphi_{i}(g(x_{i}))|\\
                                                                                        &\leq \|g(x_{i}\|.\|\sum_{i=1}^{n}c_{i}\varphi_{i}\|_{A^{*}}\\
                                                                                         &\leq \|g\|_{\infty,A}.\|\sum_{i=1}^{n}c_{i}\varphi_{i}\|_{A^{*}}.
\end{align*} 
Thus 
$$
g\in C_{BSE}(\Delta(C_{0}(X, A)).
$$
  Now, assume that $\sigma\in C_{\textup{BSE}}(\Delta(C_{0}(X, A)))$. Define
$$
\sigma_{x}(\varphi):= \sigma(x\otimes \varphi)
$$
Clearly, $\sigma_{x}\in C_{\textup{BSE}}(\Delta(A))= \hat A$, for all $x\in X$, so there exists $a_{x}\in A$ for which $\sigma_{x}= \widehat{a_{x}}$. If $Tf(x)= f(x)a_{x}$, for $f\in C_{0}(X, A)$, then $T\in M(C_{0}(X, A))$.
In fact, due to proposition \ref{pbdn} if $M> 0$ such that $\|a\|_{A}\leq M\|\hat a\|_{\textup{BSE}}$, for all $a\in A$. Then:
\begin{align*}
\|Tf(x)\|_{A} &= \|f(x)a_{x}\|_{A}\\
                         & \leq \|f(x)\|\|a_{x}\|_{A}\leq  M\|\sigma\|_{\textup{BSE}}\|f(x)\|_{A}
\end{align*}
so $Tf\in C_{0}(X, A)$.
 Now we will prove that $\sigma\in \widehat{C_{b}(X, A)}$.
Assume that $f\in C_{0}(X, A)$, then 
\begin{align*}
\sigma.\hat{f}(x\otimes\varphi) &= \sigma (x\otimes \varphi)\varphi(f(x))\\
                                                     &= \sigma_{x}(\varphi)\varphi(f(x))= \widehat{a_{x}}(\varphi)\varphi(f(x))\\
                                                      &= \widehat{a_{x}}.\widehat{f(x)}(\varphi) = \widehat{f(x)a_{x}}(\varphi)= \widehat {Tf(x)}(\varphi)\\
                                                       &= \varphi(Tf(x)) =x\otimes \varphi(Tf)= \widehat{Tf}(x\otimes \varphi)
\end{align*}
Then $\sigma. \hat f= \widehat{Tf}$, for all $f\in C_{0}(X, A)$. This follows that
$$
\sigma\in {m}(C_{0}(X,A))={M(C_{0}(X,A) \widehat)}= {C_{b}(X, A\widehat)}
$$
Hence
$$
C_{\textup{BSE}}(\Delta(C_{0}(X, A))\subseteq  {C_{b}(X, A\widehat)}
$$
Therefore $C_{0}(X, A)$ is a $\textup{BSE}$- algebra.
\end{proof}

\begin{thm}
Let $A$ be a commutative semisimple unital Banach algebra and $X$ be a locally compact Hausdorff topological space. Then  $C_{0}(X, A)$ is a  $\textup{BSE}$- algebra if and only if $A$ is a  $\textup{BSE}$- algebra.
\end{thm}
\begin{proof}
 Let  $C_{0}(X, A)$  be a $\textup{BSE}$- algebra  and $\sigma\in C_{\textup{BSE}}(\Delta(A))$. Define $\overline{\sigma}(x\otimes\varphi):= \sigma(\varphi)$, so $\overline{\sigma}\in C_{\textup{BSE}}(\Delta(C_{0}(X,A))$.
Then by using Lemma \ref{l3v}, the following is yield:
\begin{align*}
|\sum_{i=1}^{n}c_{i}\overline{\sigma}(x_{i}\otimes\varphi_{i})| &= |\sum_{i=1}^{n}c_{i}\sigma(\varphi_{i})|\\
                                                                                                         &\leq \|\sigma\|_{BSE}\|\sum_{i=1}^{n}c_{i}\varphi_{i}\|_{A^{*}}\\
                                                                                                          & \leq \|\sigma\|_{BSE}\|\sum_{i=1}^{n}c_{i} (x_{i}\otimes\varphi_{i})\|_{C_{0}(X, A)^{*}}
\end{align*}

 By hypothesis,
$$
C_{\textup{BSE}}(\Delta(C_{0}(X,A))= {M(C_{0}(X,A)\hat)}= {C_{b}(X, A\hat)}
$$
So there exists $f\in C_{b}(X, A)$ which $\overline{\sigma}= \hat f$. Consequently, for $x_{0}\in X$ the following is yield:
\begin{align*}
\sigma(\varphi) &= \overline{\sigma}(x_{0}\otimes \varphi)= \hat{f}(x_{0}\otimes \varphi)\\
                         &= \varphi(f(x_{0}))= \widehat {f(x_{0})}(\varphi)
\end{align*}
and so $\sigma= \widehat{f(x_{0})}$. Hence $\sigma\in \hat A$. Clearly, $\hat A\subseteq  C_{\textup{BSE}}(\Delta(A))$, thus
$$
 C_{\textup{BSE}}(\Delta(A))= \hat A= \widehat{M(A)}
$$
Therefore $A$ is a $\textup{BSE}$-  algebra.\\
The proof of the reverse, by applying Lemma \ref{l1cbse} is yield.
\end{proof}
\begin{lem}\label{lfc0}
Let $A$ be a commutative semisimple Banach algebra and $X$ be a locally compact Hausdorff topological space. If $f\in C_{b}(X)$ and $\sigma\in  C_{\textup{BSE}}(\Delta(A))$, then ${\hat f}\otimes \sigma\in C_{\textup{BSE}}(\Delta(C_{0}(X,A))$.
\end{lem}
\begin{proof}
Assume that $c_{i}\in\mathbb C$, $\varphi_{i}\in\Delta(A)$ and $x_{i}\in X$. Then by applying \ref{l3v}, the following is yield: 
\begin{align*}
|\sum_{i=1}^{n}c_{i}{\hat f}\otimes \sigma(x_{i}\otimes\varphi_{i})| &= |\sum_{i=1}^{n}c_{i}f(x_{i})\sigma(\varphi_{i})|\\
                                                                                       &\leq\|\sigma\|_{\textup{BSE}}\|\sum_{i=1}^{n}c_{i}f(x_{i})\varphi_{i}\|\\
                                                                                      &\leq\|\sigma\|_{\textup{BSE}}\|f\|_{\infty}\|\sum_{i=1}^{n}c_{i}x_{i}\otimes\varphi_{i}\|
\end{align*}
Thus ${\hat f}\otimes\sigma\in C_{\textup{BSE}}(\Delta(C_{0}(X,A))$.
\end{proof}
\begin{lem}
Let $A$ be a commutative semisimple unital Banach algebra and $X$ be a locally compact Hausdorff topological space. If $f\in C_{0}(X)$ and $\sigma\in  C_{\textup{BSE}}^{0}(\Delta(A))$, then ${\hat f}\otimes \sigma\in C_{\textup{BSE}}^{0}(\Delta(C_{0}(X,A))$
\end{lem}
\begin{proof}
Assume that $c_{i}\in\mathbb C$,  and $x_{i}\in X$. So for all $\epsilon>0$ there exist some compact set $F\subseteq X$ such that $|f(x)|< \epsilon$, for all $x\notin F$. Set $K:= F\times \Delta(A)$, so for each $x_{i}\otimes \varphi_{i}\notin K$, we have
\begin{align*}
|\sum_{i=1}^{n}c_{i}{\hat f}\otimes \sigma(x_{i}\otimes\varphi_{i})| &= |\sum_{i=1}^{n}c_{i}f(x_{i})\sigma(\varphi_{i})|\\
                                                                                       &\leq\epsilon \|\sum_{i=1}^{n}c_{i}f(x_{i})\varphi_{i}\|_{A^{*}}\\
                                                                                      &\leq\epsilon \|f\|_{\infty}\|\sum_{i=1}^{n}c_{i}x_{i}\otimes\varphi_{i}\|_{{C_{0}(X,A)}^{*}}\\
\end{align*}
Thus ${\hat f}\otimes\sigma\in C_{\textup{BSE}}^{0}(\Delta(C_{0}(X,A))$.
\end{proof}
\begin{thm}\label{tbse}
Let $A$ be a commutative semisimple unital Banach algebra and $X$ be a locally compact Hausdorff topological space. Then  $C_{0}(X, A)$ is a  weak-$\textup{BSE}$ algebra if and only if $A$ is a  weak- $\textup{BSE}$ algebra and $X$ is compact.
\end{thm}
\begin{proof}
Assume that $C_{0}(X, A)$ is a  weak-$\textup{BSE}$ algebra. Then 
$$
C_{\textup{BSE}}(\Delta(C_{0}(X, A)) = C_{0}(X, A\hat{)}\subseteq C_{b}(X, A\hat{)}
$$
Clearly, by an argument, as is used in \ref{l1cbse} we have
$$
C_{b}(X, A\hat{)}\subseteq  C_{\textup{BSE}}(\Delta(C_{0}(X, A))
$$

Which implies that 
$$
C_{b}(X, A\hat{)}= C_{0}(X, A\hat{)}
$$
If $f\in C_{b}(X)$, then 
$$
f\otimes 1_{A}\in {C_{b}}(X, A)= C_{0}(X, A)
$$
and so there exists some $g\in C_{0}(X, A)$ where 
$$ 
g(x)= f(x). 1_{A} ~~ (x\in X)
$$
Also, for all $\epsilon> 0$, there exists some compact set $K\subseteq X$ such that
$$
|f(x)|.\|1_{A}\|= \|f(x)1_{A}\| = \|g(x)\|_{A} < \epsilon\|1_{A}\|. 
$$
for all $x\in X\backslash K$.
Consequently, $|f(x)|< \epsilon$, for all $x\in X\backslash K$, so $f\in C_{0}(X)$. Thus $C_{b}(X)= C_{0}(X)$ and so $X$ is compact. Conversely, if $A$ is a   weak-$\textup{BSE}$ algebra, then $ C_{\textup{BSE}}(\Delta(A))= \hat A$.
So by Lemma \ref{l1cbse} $C_{0}(X, A)$ is a $\textup{BSE}$-  algebra. Also, $X$ is compact, so by applying Lemma \ref{lcc} we have
\begin{align*}
C_{\textup{BSE}}(\Delta(C_{0}(X, A))  &= C_{b}(X, A\hat{)}\\
                                                               &= C_{0}(X, A\hat{)}
\end{align*}
Therefore $C_{0}(X, A)$ is a  weak-$\textup{BSE}$ algebra.
\end{proof}
\begin{lem}\label{labed}
Let $A$ be a commutative semisimple unital Banach algebra and $X$ be a locally compact Hausdorff topological space. If  $C_{0}(X, A)$ is a  $\textup{BED}$- algebra, then $A$ is a  $\textup{BED}$- algebra.
\end{lem}
\begin{proof}
If $x_{0}\in X$, $f\in C_{0}(X)$ such that $f(x_{0})=1$ and $\sigma\in  C_{\textup{BSE}}^{0}(\Delta(A))$, then according to Lemma \ref{lfc0}, $\hat{f}\otimes \sigma\in C_{\textup{BSE}}^{0}(\Delta(C_{0}(X, A))= C_{0}(X, A\hat)$. Thus
there exists some $g\in C_{0}(X, A)$ where $\hat{f}\otimes \sigma= \hat g$. This implies that
\begin{align*}
\sigma(\varphi) &= f(x_{0})\sigma(\varphi)\\
                           &= \hat{g}(x_{0}\otimes\varphi)\\
                           &= \widehat{g(x_{0})}(\varphi)
\end{align*}
Thus $\sigma= \widehat{g(x_{0})}\in \hat A$. Therefore 
$$
C_{\textup{BSE}}^{0}(\Delta(A))\subseteq \hat A
$$
Conversely, if $a\in A$, $x_{0}\in X$ and $f\in C_{0}(X)$ where $f(x_{0})=1$, then $\hat{f}\otimes \hat{a}\in  C_{\textup{BSE}}^{0}(\Delta(C_{0}(X, A))$. So  
for all $\epsilon> 0$, there exists some compact set $K\subseteq X\otimes\Delta(A)$ where  the following inequality is holds:
$$
 |\sum_{i=1}^{n} c_{i}\hat{f}\otimes \hat{a}(x_{i}\otimes\varphi_{i})|\leq \epsilon\|\sum_{i=1}^{n} c_{i}x_{i}\otimes\varphi_{i}\|_{{C_{0}(X, A)}^{*}}
$$
where $x_{i}\otimes\varphi_{i}\notin K$.  If 
\begin{align*}
S &: ~ X\times \Delta(A)\to \Delta(A)\\
   & x\otimes \varphi\mapsto \varphi
\end{align*}
Then $S$ is continuous and $F:= S(K)$ is compact. If $\varphi_{i}\notin F$, then $x_{0}\otimes \varphi_{i}\notin K$. This implies that 
\begin{align*}
|\sum_{i=1}^{n} c_{i}\varphi_{i}(a)| &= |\sum_{i=1}^{n} c_{i} \hat{f}\otimes \hat{a}(x_{0}\otimes\varphi_{i})|\\
                                                              &\leq \epsilon\|\sum_{i=1}^{n} c_{i}x_{0}\otimes\varphi_{i}\|_{{C_{0}(X, A)}^{*}}\\
                                                              &\leq \epsilon\|\sum_{i=1}^{n} c_{i}\varphi_{i}\|_{ A^{*}}
\end{align*}
Thus $\hat a\in C_{\textup{BSE}}^{0}(\Delta(A))$ and so $\hat A\subseteq C_{\textup{BSE}}^{0}(\Delta(A))$. Therefore 
$$
C_{\textup{BSE}}^{0}(\Delta(A)) = \hat A
$$
and $A$ is a $\textup{BED}$- algebra.
\end{proof}
\begin{thm}\label{tcbe}
Let $X$ be a locally compact Hausdorff topological space and $A$ be a commutative unital Banach algebra. Then $C_{0}(X, A)$ is a $\textup{BED}$- algebra if and only $A$ is a $\textup{BED}$- algebra.
\end{thm}
\begin{proof}
If $A$ is a $\textup{BED}$- algebra, Since $A$ is unital, so $A$ is a $\textup{BSE}$- algebra. Thus by applying Lemma \ref{l1cbse}, $C_{0}(X, A)$ is a $\textup{BSE}$- algebra.  As a result
$$
C_{b}(X, A\widehat) = M(C_{0}(X, A)\widehat)= C_{\textup{BSE}}(\Delta(C_{0}(X, A))
$$ 
In another hand, $A$ is unital, so $A$ is Tauberian, then
$$
C_{0}(X, A\widehat)\subseteq C_{\textup{BSE}}^{0}(\Delta(C_{0}(X, A))
$$
If $\sigma\in  C_{\textup{BSE}}^{0}(\Delta(C_{0}(X, A))$, then there exist $g\in C_{b}(X, A)$ where $\sigma = \hat g$. Thus by applying Lemma\ref{pfc},  $g\in C_{0}(X, A)$. As a result 
$$
C_{0}(X, A\widehat)\subseteq C_{\textup{BSE}}^{0}(\Delta(C_{0}(X, A))
$$
Therefore $C_{0}(X, A)$  is a $\textup{BED}$- algebra.

  The opposite proof is based on Lemma \ref{labed}.
\end{proof}
\begin{lem}\label{lcwbe}
Let $A$ be a commutative semisimple unital Banach algebra and $X$ be a locally compact Hausdorff topological space. If  $C_{0}(X, A)$ is a  weak- $\textup{BED}$ algebra, then $X$ is compact and $A$  is a  weak- $\textup{BED}$ algebra.
\end{lem}
\begin{proof}
Since $C_{0}(X, A)$ is unital, so by applying \ref{l1},  it has $\Delta$- weak approximate identity. Then $C_{0}(X, A)$ is a $\textup{BED}$- algebra and 
$$
C_{\textup{BSE}}(\Delta(C_{0}(X, A))= C_{0}(X, A \widehat)
$$
Also according to the assumption:
$$
C_{\textup{BSE}}^{0}(\Delta(C_{0}(X, A))={M(C_{0}(X, A)\widehat)}= {C_{b}(X, A\widehat)}
$$
Thus for all $f\in C_{b}(X)$, the function $g= f\otimes 1_{A}$ is in $C_{b}(X, A)$.  This implies that  for all $\epsilon> 0$, there exists some compact set $K\subseteq X\otimes \Delta(A)$ where for all $c_{i}\in\mathbb C$ and $\psi_{i}\in X\otimes\Delta(A)\backslash K$, the following inequality
$$
 |\sum_{i=1}^{n} c_{i}\hat{g}(\psi_{i})|\leq \epsilon\|\sum_{i=1}^{n} c_{i}\psi_{i}\|_{{C_{0}(X,A)}^{*}}
$$
holds. Define
$$
\Theta: X\otimes\Delta(A)\to X
$$
where $x\otimes\varphi\mapsto x$, then $\Theta$ is continuouse. Set $F:=\Theta(K)$, so $F$ is compact, thus for all $x\in X\backslash F$ and $\varphi\in\Delta(A)$, $x\otimes\varphi\notin K$, so
\begin{align*}
|\hat{g}(x\otimes\varphi)| &= |f(x)\varphi(1_{A})|= |f(x)|\\
                                           & <\epsilon\|x\otimes\varphi\|_{{C_{0}(X,A)}^{*}} <\epsilon
\end{align*}
Then $f\in C_{0}(X)$ and 
$$
  C_{0}(X)=  C_{b}(X)
$$
Therefore $X$ is compact and 
$$
C_{\textup{BSE}}(\Delta(C_{0}(X, A))= C_{0}(X, A\widehat)
$$
Consequently $C_{0}(X, A)$ is a  $\textup{BED}$- algebra and by applying the Theorem \ref{tcbe} $A$ is a  $\textup{BED}$- algebra. Also, $A$ is unital so $A$ is a weak- $\textup{BED}$ algebra. 

\end{proof}
\begin{thm}
Let $A$ be a commutative semisimple unital Banach algebra and $X$ be a locally compact Hausdorff topological space. Then  $C_{0}(X, A)$ is a  weak-$\textup{BED}$ algebra if and only if $A$ is a  weak- $\textup{BED}$ algebra and $X$ is compact.
\end{thm}
\begin{proof}
Assume that $C_{0}(X, A)$ is a  weak-$\textup{BED}$ algebra. Then by applying lemma \ref{lcwbe}, $A$ is a  weak- $\textup{BED}$ algebra, and $X$ is compact.

   Conversely, let $A$ be a  weak- $\textup{BED}$ algebra and $X$ be a compact space. Since $A$ is unital and semisimple, so 
$$
C_{\textup{BSE}}^{0}(\Delta(A))= \widehat{M(A)}= \hat A
$$
Thus $A$ is a $\textup{BED}$- algebra and  by applying the Theorem \ref{tcbe} $C_{0}(X, A)$ is a $\textup{BED}$- algebra. Also, $X$ is compact, so the following is the yield:
\begin{align*}
C_{\textup{BSE}}^{0}(\Delta(C_{0}(X, A)) &= C_{0}(X, A\widehat)\\
                                                                       &= C_{b}(X, A \widehat)\\
                                                                       & =M(C_{0}(X, A)\widehat)
\end{align*}
Therefore $C_{0}(X, A)$ is a  weak-$\textup{BED}$ algebra.
  
\end{proof}



\end{document}